\documentclass[10pt]{amsart}  
\usepackage{amsmath,amssymb, mathtools}
\usepackage{cancel}
\usepackage[colorlinks=true,urlcolor=blue, citecolor=red,linkcolor=blue,linktocpage,pdfpagelabels, bookmarksnumbered,bookmarksopen]{hyperref}
\usepackage[hyperpageref]{backref}
\usepackage{cleveref}

\newcommand{\R}{\mathbb{R}}
\newcommand{\N}{\mathbb{N}}
\newcommand{\eps}{\varepsilon}

\newtheorem{theorem}{Theorem}
\newtheorem{lemma}[theorem]{Lemma}
\newtheorem{corollary}[theorem]{Corollary}
\newtheorem{proposition}[theorem]{Proposition}

\theoremstyle{definition}

\newtheorem{definition}[theorem]{Definition}
\newtheorem{example}[theorem]{Example}

\usepackage{xcolor}  

\newcommand{\brac}[1]{\left ( #1 \right)}

\newcommand{\chgd}[1]{{ #1}}
\newcommand{\chgdtwo}[1]{{#1}}

\numberwithin{theorem}{section} \numberwithin{equation}{section}

\author{Josh Ascher}
\thanks{joa71@pitt.edu}

 \author{Armin Schikorra}
 \thanks{armin@pitt.edu}
\address{Department of Mathematics, University of Pittsburgh, Pittsburgh, PA 15261, USA.}

\title{Carnot-\chgd{Carath\'eodory} and \chgd{Kor\'anyi}-Geodesics in the Heisenberg Group}
\date{Fall 2021}
\begin{document}

\begin{abstract}
This paper is part of an undergraduate research project. We discuss the Heisenberg group $\mathbb{H}_1$, the three-dimensional space $\mathbb{R}^3$ equipped with one of two equivalent metrics, the \chgd{Kor\'anyi}- and Carnot-\chgd{Carath\'eodory} metric. We show that the notion of length of curves for both metrics coincide, and that shortest curves, so-called geodesics, exist.
\end{abstract}

\maketitle 
\tableofcontents

\newenvironment{myindentpar}[1]%
 {
%  \begin{list}{}%
        %  {\setlength{\leftmargin}{#1}}%
        %  \item[]%
 }
 {
%  \end{list}
}

\section{Introduction}
The Heisenberg group $\mathbb{H}_1$ is a subject of intensive study, as a special case of sub-Riemannian manifolds or Carnot groups, see \cite{G96} or \cite{CDPT07}.

From the point of view of Analysis, $\mathbb{H}_1$ consists of all the points $p = (p^1,p^2,p^3) \in \mathbb{R}^3$, where $\mathbb{R}^3$ denotes the usual Euclidean three-dimensional space. However, the distance between two points $p,q \in \mathbb{R}^3$ is given by a non-Euclidean metric $d(p,q)$. Actually, there are two typical metrics used in the Heisenberg group $\mathbb{H}_1$, and we begin by describing the first one, the \emph{Carnot-\chgd{Carath\'eodory}}-metric $d_{cc}(p,q)$ of $\mathbb{H}_1$:
Take any (for now continuously differentiable) curve $\gamma: [0,1] \to \R^3$ with $\gamma(0) = q$ and $\gamma(1) = p$. From calculus we know that the length of a curve is given by 
\begin{equation}\label{eq:lengtheuclidean}
\mathcal{L}(\gamma) = \int_{[0,1]} |\dot{\gamma}(t)| dt,
\end{equation}
where $\dot{\gamma}$ denotes the derivative of $\gamma$. If we consider the minimal possible length of curves $\gamma: [0,1] \to \R^3$ that are continuously differentiable and connect $p$ to $q$ in the sense that  $\gamma(0) = p$ and $\gamma(1) = q$, then one can show that this minimal length is exactly the Euclidean distance $|p-q|$,
\[
|p-q|_{\R^3} = \inf_{\gamma \in X(p,q)} \mathcal{L}(\gamma)
\]
where 
\[
X(p,q) = \left \{\gamma: [0,1] \to \mathbb{R}^3: \quad \text{continuously differentiable, }\gamma(0) = p,\, \gamma(1) = q \right \}.
\]
The Carnot-\chgd{Carath\'eodory} metric is also the infimum of the lengths of curves connecting $p$ and $q$, however those curves have to be \emph{horizontal}, meaning that $\dot{\gamma}(t)$ has to belong to the horizontal space $H_{\gamma(t)} \mathbb{H}_1$ for each $t \in (0,1)$, which is spanned by the vectors 
\[
H_{p} \mathbb{H}_1 = {\rm span} \left \{ 
\left (\begin{array}{c}
1\\
0\\
2p^2
\end{array}\right )
, \left (\begin{array}{c}
0\\
1\\
-2p^1\end{array}\right )
\right \}.
\]
That is for each $t \in (0,1)$ there must be some $\lambda_1(t)$ and $\lambda_2(t)$ such that 
\[
\dot{\gamma}(t) = \lambda_1(t) \left (\begin{array}{c}
1\\
0\\
2\gamma^2(t)
\end{array}\right ) + \lambda_2(t)\left (\begin{array}{c}
0\\
1\\
-2\gamma^1(t)\end{array}\right ),
\]
or, taking $\lambda_1(t) = \dot{\gamma}^1(t)$ and $\lambda_2(t) = \dot{\gamma}^2(t)$, equivalently,
\begin{equation}\label{eq:horizontal}
\dot{\gamma}^3(t) =  -2 \brac{\gamma^1(t) \dot{\gamma}^2(t) - \gamma^2(t) \dot{\gamma}^1(t)} \quad \forall t \in (0,1).
\end{equation}
For such curves we define the length
\[
\mathcal{L}_{cc}(\gamma) := \int_{[0,1]} \sqrt{|\lambda_1(t)|^2 + |\lambda_2(t)|^2} dt \equiv \int_{[0,1]} \sqrt{|\dot{\gamma}_1(t)|^2+|\dot{\gamma}_2(t)|^2} dt.
\]
The Carnot-\chgd{Carath\'eodory} length $d_{cc}(p,q)$ is then given by 
\begin{equation}\label{eq:dccpq}
d_{cc}(p,q) = \inf_{\gamma \in Y(p,q)} \mathcal{L}_{cc}(\gamma)
\end{equation}
where 
\[
Y(p,q) = \left \{\gamma: [0,1] \to \mathbb{R}^3: \quad \text{continuously differentiable, }\gamma(0) = p,\, \gamma(1) = q,\, \text{\eqref{eq:horizontal} holds} \right \}.
\]
Observe that this is very similar to curves $\gamma$ into a Riemannian manifold $\mathcal{M} \subset \R^3$: any differentiable curve $\gamma: [0,1] \to \mathcal{M}$ satisfies $\gamma(t) \in T_{\gamma(t)} \mathcal{M}$, where $T_{p} \mathcal{M}$ is the tangent space of the manifold $\mathcal{M}$, and if we want to find the distance between two points $p$ and $q$ on the manifold, it makes sense to define this distance as the minimal length of curves tangent to the manifold at every point and connecting $p$ and $q$.
So from this perspective, the Heisenberg group is $\R^3$ with a ``strange'' tangent plane distribution (and since it is strange we call it horizontal plane distribution instead) -- the strangeness of the Heisenberg group is that its horizontal plane distribution cannot be written as a tangent space of any manifold $\mathcal{M}$, that is the horizontal plane distribution is not integrable in the sense of the Frobenius' theorem. Here is actually where the ``group'' of the Heisenberg group enters, the vectors spanning the horizontal space $H_{p} \mathbb{H}_1$ are left-invariant vector fields for a group structure -- but we will not pursue this point of view further here.

It is known that for each $p,q \in \R^3$ the infimum in \eqref{eq:dccpq} is attained, i.e. there exists a shortest curve $\gamma$, called \emph{geodesic} such that 
\[
\mathcal{L}_{cc} (\gamma) = d_{cc}(p,q),
\]
see e.g. Haj\l{}asz-Zimmerman \cite[(1.3)]{HZ}. In particular between any two points $p,q \in \R^3$ there exist horizontal curves. \chgd{Let us remark that for more general sub-Riemannian geometry it a very deep result, called Chow–Rashevskii theorem, that $d_{cc}(p,q)$ is even \emph{finite} for all points $p,q$, cf. \cite{M02}.}

While the above notion of distance $d_{cc}(p,q)$ is attractive from a geometric point of view, it is not easily computable (given $p$ and $q$ we first need to find the shortest curve $\gamma$ between then, then compute its length).

The other metric we want to consider, the \emph{\chgd{Kor\'anyi}-metric}, is \chgd{much} easier to compute. It simply is given by
\[
d_{K}(p,q) :=(\brac{|p_1-q_1|^2 + |p_2-q_2|^2}^2 + \left |p_3 - q_3 + 2(p_2 q_1 - p_1 q_2)\right |^2)^{\frac{1}{4}}
\]
There is also a more group-theoretic motivation for $d_{K}(p,q) = \|p^{-1} \ast q\|_{\mathbb{H}_1}$, but we will also not pursue this aspect further here, we refer the interested reader to \cite{CDPT07}.

Any metric space naturally is equipped with a notion of length of curves, see \Cref{def:length}, which gives us the notion of a \chgd{Kor\'anyi}-length $\mathcal{L}_K(\gamma)$.

We will first prove the following result.

\begin{theorem}\label{th:existencegeodesic}
Let $p,q \in \R^3$. Then there exists a shortest continuous curve (i.e. a geodesic) $\gamma: [0,1] \to \R^3$, $\gamma(0)=p$, $\gamma(1) =q$ such that 
\[
\mathcal{L}_{K}(\gamma) = \inf_{\tilde{\gamma} \in \tilde{X}(p,q)} \mathcal{L}_{K}(\tilde{\gamma}),
\]
where
\[
\tilde{X}(p,q) \coloneqq \left \{\gamma: [0,1] \to \mathbb{R}^3: \quad \text{continuous, }\gamma(0) = p,\, \gamma(1) = q \right \}.
\]
Observe the difference to $X(p,q)$ above is that curves do not need to be differentiable.
\end{theorem}

The above theorem follows from a general principle using the \chgd{Arzel\'a}-Ascoli theorem and holds true in much more generality. 

More specifically to the Heisenberg group we will show that although the metric $d_{K}$ differs from $d_{cc}$, the \chgd{Kor\'anyi}-length $\mathcal{L}_{K}$ equals the Carnot-\chgd{Carath\'eodory} length $\mathcal{L}_{cc}$.

\begin{theorem}\label{th:lengthsame}
Let $\gamma: [0,1] \to \R^3$ be twice continuously differentiable. If $\gamma$ is horizontal (i.e. \eqref{eq:horizontal} holds) and $\mathcal{L}_{cc}(\gamma) < \infty$ then $\mathcal{L}_{K}(\gamma) < \infty$ and we have
\[
\mathcal{L}_{K}(\gamma)  = \mathcal{L}_{cc}(\gamma).
\]
\end{theorem}

From \Cref{th:lengthsame} we actually can conclude that $(\R^3,d_{K})$ is not a length space: By the definition of length of a curve in a metric space $(X,d)$, see \Cref{def:length}, for any $p,q$ and any curve $\gamma: [0,1] \to X$, $\gamma(0) = p$, $\gamma(1) = q$ we have the inequality
\[
 \mathcal{L}(\gamma) \geq d(p,q).
\]
If for any $p,q \in X$ there exists a curve $\gamma: [0,1] \to X$, $\gamma(0) = p$, $\gamma(1) = q$ such that we have equality 
\[
 \mathcal{L}(\gamma) = d(p,q),
\]
then we call $X$ a length space. The following example shows that $(\R^3,d_{K})$ is not a length space (this is in contrast to the Carnot-\chgd{Carath\'eodory} metric where the corresponding equality holds by definition \eqref{eq:dccpq}).

\begin{example}
The following is the shortest curve between $p:= (0,0,0)$ and $q:= (0,0,\frac{1}{4\pi})$
\[
\gamma(t) = \left ( 
\begin{array}{c}
(1-\cos(2\pi t))\\ 
\sin(2\pi t)\\
\frac{1}{4\pi} (t-\frac{\sin(2\pi t)}{2\pi})
\end{array} \right ).
\]
See \cite[Theorem 2.1]{HZ}. It can be checked by a direct computation that $\mathcal{L}_{K}(\gamma) = \mathcal{L}_{cc}(\gamma) > d_{K}(p,q)$
\end{example}

The outline of the remaining part of the paper is as follows: in \Cref{s:prelimmetric} we discuss preliminary results on metric spaces, in particular \chgd{Arzel\'a}-Ascoli's  theorem. 
In \Cref{s:horizontalcurves} we discuss properties of horizontal curves that we need for both theorems.
In \Cref{s:heisenbergcurves} we establish the existence of shortest curves with respect to $\mathcal{L}_K$ in the Heisenberg group. In \Cref{s:proofofmainthm} we prove \Cref{th:lengthsame}. Let us remark that the results in this work are probably well-known to experts, the purpose of this paper is to provide a detailed account making this exciting field accessible to non-experts, students and early career researchers.
\chgd{
\subsection*{Acknowledgment} The authors would like to thank the anonymous referee for valuable suggestions on the article.}

\section{Some preliminary Statements from Analysis: Metric Spaces}\label{s:prelimmetric}
Let $X$ be a metric space with metric $d$. A \emph{curve} $\gamma$ is simply a continuous map $\gamma: I \to X$, where $I=[a,b]$ is any closed finite interval. 

We say that a curve $\gamma: [a,b] \to X$ connects two points $p,q \in X$ if $\gamma(a) = p$ and $\gamma(b) = q$.

We now want to define the length of a curve $\gamma: [a,b] \to X$, however observe that $\gamma$ may not be differentiable. Indeed, we may not even know what differentiability of $\gamma$ means since $X$ is not a linear space! So a formula such as \eqref{eq:lengtheuclidean} does not make sense. But recall from Calculus how we obtained the formula \eqref{eq:lengtheuclidean}, we used polygonal approximation of a curve. We will do the same in metric spaces.

\begin{definition}[Partition]\label{def:partition}
Given an interval $[a,b]$, a partition of size $n$ is the set $\{x_0, x_1, \dots, x_n\}$ where
\[
a=x_0 < x_1 < \dots < x_n = b
\]
\end{definition}

With the notion of partition we can ``approximate'' curves by a \chgd{discrete path} through the points $\gamma(a),\gamma(x_1),\ldots,\gamma(b)$. Then we use the metric to define the length of these ``polygon''-lines.

\begin{definition}[Length of curve]\label{def:length}
Given a metric space $(X,d)$ and a curve $\gamma: [a,b] \to X$. The length of $\gamma$ is given by
\[
\mathcal{L}(\gamma) = \sup_{p \in P} \sum_{\chgd{i= 1}}^n d(\gamma(t_i),\gamma(t_{i-1})),
\]
where the supremum is taken over all partitions $p$ of $[a,b]$ (i.e. $P$ is the collection of all partitions of $[a,b]$).
\end{definition}
Observe that the length of a curve $\mathcal{L}(\gamma)$ is always nonnegative, indeed since $\{a,b\}$ is a partition of $[a,b]$, we have 
\begin{equation}\label{eq:lengthab}
\mathcal{L}(\gamma) \geq d(\gamma(a),\gamma(b)).
\end{equation}
In general, even if $d(\gamma(a),\gamma(b)) < \infty$ the length $\mathcal{L}(\gamma)$ could be $+\infty$. We call any curve $\gamma$ with finite length $\mathcal{L}(\gamma) < \infty$ \emph{rectifiable}.

It is worth noting the following \chgd{
\begin{lemma}\label{la:absolutecontinuity}
Given a metric space $(X,d)$, let $\gamma: [a,b] \to X$ be a curve of finite length, $\mathcal{L}(\gamma) < \infty$. Then for any $s_0 \in [a,b]$, the restricted curves
\[
 \gamma \Big |_{[s_0,b]}: [s_0,b] \to X, \quad [s_0,b] \ni t \mapsto \gamma(t)
\]
and 
\[
 \gamma \Big |_{[a,s_0]}: [a,s_0] \to X,\quad [a,s_0] \ni t \mapsto \gamma(t)
\]
are curves of finite length. Moreover 
\[
 [a,b]\ni s \mapsto \mathcal{L}\left(\gamma \Big |_{[a,s]}\right)
\]
and 
\[
 [a,b]\ni s \mapsto \mathcal{L}\left(\gamma \Big |_{[s,b]}\right)
\]
are continuous monotone increasing maps.
\end{lemma}
\begin{proof}
Finiteness and monotonicity are easy to obtain from the definition of the curve. For the continuity, we observe that for $a \leq s_1 < s_2 \leq b$
\[
 \mathcal{L}\left(\gamma \Big |_{[a,s_2]}\right)-\mathcal{L}\left(\gamma \Big |_{[a,s_1]}\right)= \mathcal{L}\left(\gamma \Big |_{[s_1,s_2]}\right)
\]
and
\[
 \mathcal{L}\left(\gamma \Big |_{[s_1,b]}\right)-\mathcal{L}\left(\gamma \Big |_{[s_2,b]}\right)= \mathcal{L}\left(\gamma \Big |_{[s_1,s_2]}\right).
\]
So what we need to show is that for any $\eps > 0$ and any $s_1 \in [a,b]$ there exists $\delta > 0$ such that 
\[
 \mathcal{L}\left(\gamma \Big |_{[s_1,s_2]}\right) < \eps \quad \forall s_2: |s_1-s_2| < \delta.
\]
Fix $\eps > 0$ and $s_1 \in [a,b]$. By continuity of $\gamma$ we find $\delta_1  > 0$ such that 
\begin{equation}\label{eq:continuityasldkjasd}
 d(\gamma(\tilde{s}),\gamma(\tilde{t})) < \eps \quad \forall |\tilde{s}-s_1|,\ |\tilde{t}-s_1| < \delta_1.
\end{equation}
Since $\mathcal{L}(\gamma) < \infty$ there exists a partition 
\[
 a = t_0 < t_1 \ldots < t_n = b
\]
such that 
\[
 \mathcal{L}(\gamma) - \eps \leq \sum_{i = 1}^n d(\gamma(t_i),\gamma(t_{i-1})).
\]
Set 
\[
 \delta_2 := \inf_{i=1,\ldots,n} |t_i-t_{i-1}|.
\]
Set $\delta := \min\{\delta_1,\delta_2\}$ and fix any $s_2 \in [a,b]$ with $|s_1-s_2|<\frac{\delta}{2}$. 

W.l.o.g. $s_1 < s_2$. We then may assume that $t_{i_0-1}<s_1 < t_{i_0} < s_2<t_{i_0+1}$ for some $i_0 \in \N$ (all other cases follow by an easy adaptation). We now consider the new partition $\tilde{t}_i$,
\[
 \tilde{t_i} = \begin{cases}
                t_i \quad &i \leq i_0-1\\
                s_1 \quad &i=i_0\\
                t_{i_0} \quad &i=i_0+1\\
                s_2 \quad &i=i_0 + 2\\
                t_{i-2} \quad &i \geq i_0+3.
               \end{cases}
\]
Then, by triangular inequality,
\begin{equation}\label{eq:lengthmeps}
 \mathcal{L}(\gamma) - \eps \leq \sum_{i = 1}^{n+2} d(\gamma(\tilde{t}_i),\gamma(\tilde{t}_{i-1})).
\end{equation}
Now let $s_1 = r_0 < r_1 < \ldots = r_m = s_2$ be any partition of $[s_1,s_2]$.
Then 
\[
\begin{split}
 \sum_{j = 0}^{m} d(\gamma(r_j),\gamma(r_{j-1})) =&  \sum_{i \neq i_0+1,i_0+2} d(\gamma(\tilde{t}_i),\gamma(\chgdtwo{\tilde{t}_{i-1}}))+  \sum_{j = 0}^{m} d(\gamma(r_j),\gamma(r_{j-1})) \\
 &- \sum_{i =1}^{n+2} d(\gamma(\tilde{t}_i),\gamma(\chgdtwo{\tilde{t}_{i-1}}))\\
 &+d(\gamma(\tilde{t}_{i_0+1}),\gamma(\tilde{t}_{i_0}))+d(\gamma(\tilde{t}_{i_0+2}),\gamma(\tilde{t}_{i_0+1}))
 \end{split}
\]
Since we can combine the partitions $\tilde{t}_i$, $i \neq i_0+1,i_0+2$ with $r_j$ to obtain a partition of $[a,b]$, we have by the definition of length,
\[
\sum_{i \neq i_0+1,i_0+2} d(\gamma(\tilde{t}_i),\gamma(\chgdtwo{\tilde{t}_{i-1}}))+  \sum_{j = 0}^{m} d(\gamma(r_j),\gamma(r_{j-1})) \leq \mathcal{L}(\gamma).
\]
By \eqref{eq:lengthmeps} we have 
\[
 - \sum_{i =1}^{n+2} d(\gamma(\tilde{t}_i),\gamma(\chgdtwo{\tilde{t}_{i-1}})) \leq -\mathcal{L}(\gamma) + \eps.
\]
By \eqref{eq:continuityasldkjasd} which we can apply since $s_1 < t_{i_0} < s_2$ and thus $|s_1 -s_2|$, $|t_{i_0}-s_2| < \delta_1$,
\[
 d(\gamma(\tilde{t}_{i_0+1}),\gamma(\tilde{t}_{i_0}))+d(\gamma(\tilde{t}_{i_0+2}),\gamma(\tilde{t}_{i_0+1}))
 =d(\gamma(t_{i_0}),\gamma(s_1))+d(\gamma(s_2),\gamma(t_{i_0})) \leq 2\eps.
\]
So we have shown
\[
\begin{split}
 \sum_{j = 0}^{m} d(\gamma(r_j),\gamma(r_{j-1})) \leq 3\eps.
 \end{split}
\]
This holds for any partition $(r_j)$ of $[s_1,s_2]$ and thus 
\[
 \mathcal{L}\left(\gamma \Big |_{[s_1,s_2]}\right) < \chgdtwo{3}\eps.
\]
\chgdtwo{Since $\eps$ was arbitrary,} we can conclude.
\end{proof}
}

For simplicity, we will often restrict our attention to curves defined on $I = [0,1]$, which we can do without loss of generality. Indeed any curve 
\[
\gamma: [a,b] \to X
\]
can be reparametrized to a curve 
\[
\tilde{\gamma}: [0,1] \to X
\]
by simply setting
\[
\tilde{\gamma}(t) := \gamma(tb + (1-t) a).
\]
Similarly any curve $\gamma: [0,1] \to X$ can be reparametrized to a curve $\tilde{\gamma}: [a,b] \to X$. The length of the curve $\gamma$ and $\tilde{\gamma}$ above are the same, $\mathcal{L}(\gamma)= \mathcal{L}(\tilde{\gamma})$. Indeed, the length of curves is invariant under reparametrization.

\chgd{
\begin{definition}[Reparametrization]
Let $\gamma: [a,b] \to X$ be a curve. Let $\tau: [c,d] \to [a,b]$ be a continuous bijection with continuous inverse (i.e. a homeomorphism) such that $\tau(c) = a$ and $\tau(d) = b$. Then, $\tau$ is a reparametrization of $\gamma$.
\end{definition}

\begin{lemma}\label{la:reparamlength}
Let $\gamma: [a,b] \to X$ be a curve and $\tau: [c,d] \to [a,b]$ be a reparametrization.  Then if we set $\tilde{\gamma}(t) := \gamma(\tau(t))$ we get that $\tilde{\gamma}: [c,d] \to X$ is a curve and 
\[
\mathcal{L}(\gamma) = \mathcal{L}(\tilde{\gamma})
\]
\end{lemma}
}
We leave the proof as an exercise, but observe that $\tau$ maps any partition for $[c,d]$ into a partition of $[a,b]$, and $\tau^{-1}$ maps any partition of $[a,b]$ into a partition of $[c,d]$.

Now we want to find geodesics, i.e. shortest curves between two points $p$ and $q$ in $X$. A curve $\gamma: I \to X$ is called \chgd{the} \emph{shortest curve} or (minimizing) geodesic from $p$ to $q$ if it connects $p$ and $q$ and for any other curve $\tilde{\gamma} : \tilde{I} \to X$ which connects $p$ and $q$ we have we have $\mathcal{L}(\gamma) \leq \mathcal{L}(\tilde{\gamma}).$ 

In general metric spaces $X$ there is no reason that there exists such a shortest curve $\gamma$. As a side-note a shortest curve in general is not unique\chgd{:} think of the many shortest curves connecting the north pole and the south pole of a sphere. In order to conduct in the following chapters our analysis of the Heisenberg group, we conclude this section with a few important notions and facts on maps (possibly) on metric spaces.

The first result from Analysis is the \chgd{Arzel\'a}-Ascoli theorem -- the proof can be found in essentially all Advanced Calculus books.
Recall that a set $E \subset X$ is compact, if any sequence $(x_n)_{n \in \N} \subset E$ has a subsequence $(x_{n_i})_{i \in \N}$ and a point $x \in E$ such that $d(x_{n_i},x) \xrightarrow{i \to \infty} 0$.

\begin{theorem}[\chgd{Arzel\'a}-Ascoli]\label{th:arzelaascoli}
Let $(X,d)$ be a metric space and \chgd{$E \subset X$ be compact.} Assume there is a sequence of maps $\gamma_k: [0,1] \to E$ which are 
\emph{equicontinuous}, i.e. for any $\eps > 0$ there exists $\delta > 0$ such that $\sup_{k \in \N} d(\gamma_k(t),\gamma_k(s)) < \varepsilon$ for all $s,t \in [0,1]$ with $|s-t| < \delta$.

Then, there exists a subsequence $(\gamma_{k_i})_{i \in \N}$ and a continuous limit function $\gamma: [0,1] \to X$ such that $\gamma_{k_i}$ uniformly converge to $\gamma$ in the sense that 
\[
\sup_{t \in [0,1]} d(\gamma_{k_i}(t),\gamma(t)) \xrightarrow{i \to \infty} 0.
\]
\end{theorem}

We will use later that uniform Lipschitz continuity implies equicontinuity. Namely if there exists $\Lambda > 0$ such that 
\[
\sup_{k \in \N} d(\gamma_{k}(s),\gamma_{k}(t)) \leq \Lambda |s-t| \quad \text{for all $s,t \in [0,1]$}
\]
then the equicontinuity condition in Theorem~\ref{th:arzelaascoli} is satisfied.

We now show that any curve with finite length can be parametrized so that it is Lipschitz continuous (so curves with uniformly bounded length are uniformly Lipschitz continuous, and thus equicontinuous).

\begin{proposition}[Monotone Reparametrization]\label{pr:reparam}
Let $\gamma: [a,b] \to X$ be a curve of finite length,  $\mathcal{L}(\gamma) < \infty$. 

Then $\gamma$ admits a Lipschitz reparameterization in the following sense. 

There exists $\tilde{\gamma}: [0,1] \to X$ with the following properties \begin{itemize} \item $\tilde{\gamma}(0) = \gamma(a)$ and $\tilde{\gamma}(1) = \gamma(b)$\item $\tilde{\gamma}([0,1]) = \gamma([a,b])$ (in the sense of sets in $X$) \item $\mathcal{L}(\gamma) =\mathcal{L}(\tilde{\gamma})$,
\item $|\tilde{\gamma}(s)-\tilde{\gamma}(t)| \leq \mathcal{L}(\gamma) |s-t| \quad \forall s,t \in [0,1].$
\end{itemize}
\end{proposition}

%\begin{myindentpar}{1cm}

\begin{proof}
Without loss of generality, $[a,b] = [0,1]$. 
Let $\gamma : [0,1] \to (X,d)$ be a curve of finite length. 

Define $\tau(t) := \mathcal{L}(\gamma |_{[0,t]}): [0,1] \to [0,\mathcal{L}(\gamma)]$, which by \Cref{la:absolutecontinuity} is continuous and monotone increasing.

\chgd{We would like to set $\hat{\gamma} := \gamma \circ \tau^{-1}: [0,\mathcal{L}(\gamma)] \to X$. The issue is that $\tau$ may not be strictly monotone, so $\tau$ may not be invertible.

However $\hat{\gamma}$ is still well-defined. Observe that if for some $0 \leq r\leq \tilde{r} \leq 1$ we have $\tau(r)=\tau(\tilde{r})$, then 
\[
 0 = \mathcal{L}(\gamma |_{[0,\tilde{r}]})-\mathcal{L}(\gamma |_{[0,r]})=\mathcal{L}(\gamma |_{[r,\tilde{r}]}),
\]
that is $\mathcal{L}(\gamma |_{[r,\tilde{r}]})=0$ and from the definition of the length $\mathcal{L}$ we conclude that $d(\gamma(s),\gamma(t)) = 0$ for all $s,t \in [r,\tilde{r}]$. 

That is $\tau(r)=\tau(\tilde{r})$ implies that $\gamma$ is constant on $[\tilde{r},r]$, in particular $\gamma(r) = \gamma(\tilde{r})$.

So we can still define $\hat{\gamma} := \gamma \circ \tau^{-1}$ in the following sense: for a given $t \in [0,\mathcal{L}(\gamma)]$ take any $r \in [0,1]$ such that $\tau(r) = t$. Such a $r$ exists by the intermediate value theorem since $\tau$ is continuous, $\tau(0) = 0$ and $\tau(1) = \mathcal{L}(\gamma)$. Then we set
\[
 \hat{\gamma}(t) := \gamma(r).
\]
If we were to pick any other $\tilde{r}$ with $\tau(\tilde{r}) = t$ then by the above observation we have $\gamma(r) = \gamma(\tilde{r})$ and $\hat{\gamma}(t)$ still has the same value.

We now claim that $\hat{\gamma}$ is continuous. Fix $t_0 \in [0,\mathcal{L}(\gamma)]$ and $\eps > 0$. Take $R \subset [0,1]$ such that $\tau(r) = t_0$ for all $r \in R$. By the above observation, whenever $r,\tilde{r} \in R$ we have $[r,\tilde{r}] \subset R$. On the other hand if $(r_k)_{k \in \N} \subset [0,1]$ such that $\tau(r_k) = t_0$ for all $k \in \N$ then if $r = \lim_{k \to \infty} r_k$ we have $\tau(r) = t_0$, by continuity of $\tau$. Combining this with monotonicity of $\tau$ we find that for some $r_0 \leq r_1$
\[
 R = [r_0,r_1], \quad \text{and} \quad \tau(r) < t_0 \quad\text{if $r < r_0$}, \quad \text{and} \quad \tau(r) > t_0 \quad \text{if $r > r_1$}.
\]
By continuity of $\gamma$, there exists an $\delta_1 > 0$ such that $|\gamma(r)-\hat{\gamma}(t_0)|<\eps$ whenever $r \in (r_0-\delta_1,r_1+\delta_1)$. Let now $\delta_2 := \min\{\tau(r_0)-\tau(r_0-\delta_1),\tau(r_1+\delta_1)-\tau(r_0)) > 0$. 
Recall that $t_0 = \tau(r_0) = \tau(r_1)$. So whenever $t$ satisfies $|t-t_0| < \delta_2$ then we have $t \in (\tau(r_0-\delta_1),\tau(r_1+\delta_1))$, and thus by monotonicity, $t \chgdtwo{\in \tau}(r_0-\delta_1,r_1+\delta_1)$ which implies that $|\hat{\gamma}(t) -\hat{\gamma}(t_0)| < \eps$. That is, we have shown continuity of $\hat{\gamma}$. }

\chgd{With the same observation as above, it is now not too difficult to show that $\mathcal{L}(\gamma) =\mathcal{L}(\hat{\gamma})$ -- since the only points where $\tau$ is not invertible are points where no length is added.}
\chgdtwo{Indeed, let $0=r_0 < r_1 < \ldots < r_n =1$ be a partition of $[0,1]$. 
Set $t_0=0$ and $t_n = \mathcal{L}(\gamma)$ and set $t_i := \tau(r_i)$ for $i=1,\ldots,n-1$. Then $\hat{\gamma}(t_i) = \gamma(r_i)$. By monotonicity of $\tau$ we have $0 =t_1 \leq t_2 \leq \ldots \leq t_n = \mathcal{L}(\gamma)$. It might happen that we have equality $t_{i} = t_{i-1}$ but then $\tau(r_{i}) = \tau(r_{i-1})$ which by the argument above means $\hat{\gamma}(t_{i}) = \hat{\gamma}(t_{i-1})$ and thus $d(\hat{\gamma}(t_i),\hat{\gamma}(t_{i-1}))=0$. Consequently we have 
\[
\sum_{i=1}^n d(\gamma(r_i),\gamma(r_{i-1})) = \sum_{i=1}^n d(\hat{\gamma}(t_i),\hat{\gamma}(t_{i-1})) \leq \mathcal{L}(\hat{\gamma}).
\]
Taking the supremum of all partitions of $[0,1]$ we have 
\begin{equation}\label{eq:lsdfjlksdfj1}
\mathcal{L}(\gamma) \leq \mathcal{L}(\hat{\gamma}).
\end{equation}
For the other direction let $0 = t_0 < t_1 < \ldots < t_n= \mathcal{L}(\gamma)$ be any partition of
%(we can chat in the chat!)  
$[0,\mathcal{L}(\gamma)]$. 
We now create a new partition $0 = r_0 < \ldots < r_i < \ldots < r_n = 1$ such that $\tau(r_i) = t_i$, and thus by the definition of $\hat{\gamma}$, $\gamma(r_i) = \hat{\gamma}(t_i)$. We set $r_0 := 0$ and $r_n :=1$. We define $r_i$ to be any $r_i \in (0,1)$ such that $\tau(r_i) = t_i$, this choice of $r_i$ may not be unique but from the intermediate value theorem at least one such $r_i$ must exists.  Since $t_{i-1} < t_i$ for all $i$, from the monotonicity of $\tau$ we conclude that $r_{i-1} < r_i$ for all $i$, and thus $0=r_0 < r_1 < \ldots < r_n = 1$ is the desired new partition of $[0,1]$. We then have
\[
\sum_{i=1}^n d(\hat{\gamma}(t_i),\hat{\gamma}(t_{i-1})) = \sum_{i=1}^n d(\gamma(r_i),\gamma(r_{i-1})) \leq \mathcal{L}(\gamma).
\]
Taking the supremum over all partitions of  $[0,\mathcal{L}(\gamma)]$ we conclude 
\begin{equation}\label{eq:lsdfjlksdfj2}
\mathcal{L}(\hat{\gamma}) \leq \mathcal{L}(\gamma).
\end{equation}
Together, \eqref{eq:lsdfjlksdfj1} and \eqref{eq:lsdfjlksdfj2} imply 
\[
\mathcal{L}(\hat{\gamma}) = \mathcal{L}(\gamma).
\]

}
Next, we observe that the definition of the length of a curve implies 
$$d(\gamma(t) , \gamma(s)) \overset{\eqref{eq:lengthab}}{\leq} \mathcal{L}\left(\gamma \ |_{[s,t]}\right) = \left |\mathcal{L}\left(\gamma|_{[0,t]}\right) - \mathcal{L}\left(\gamma|_{[0,s]}\right)\right |=|\tau(t)-\tau(s)|.$$
%(first inequality: definition of length as supremum over partitions; second inequality: supremum over sample points, refiniement property)
Let $\hat{s}, \hat{t} \in [0,\mathcal{L}(\gamma)]$ and take any $s,t \in [0,1]$ such that $\tau(s) = \hat{s}$, $\tau(t) = \hat{t}$. Then 
$$d(\hat{\gamma}(\hat{t}), \hat{\gamma}(\hat{s})) d(\gamma(t), \gamma(s)) \leq |\tau(t)-\tau(s)| = |\hat{t}-\hat{s}|.$$

Thus, $\hat{\gamma}$ is Lipschitz continuous, albeit with the wrong constant, which is easy to fix.

Set
\[
\tilde{\gamma}(s) := \hat{\gamma}(\mathcal{L}(\gamma)s), \quad s \in [0,1].
\]
Then we have 
\[
d(\tilde{\gamma}(s),\tilde{\gamma}(t)) \leq \mathcal{L}(\gamma)|s-t| \quad \forall s,t \in [0,1].
\]
\end{proof}

The \chgd{Arzel\'a}-Ascoli theorem, \Cref{th:arzelaascoli}, will play a crucial role in constructing a \emph{candidate} for a shortest curve in the Heisenberg group. Another important ingredient is the following lower semicontinuity of the length.

\begin{proposition}[Lower semicontinuity of the length functional]\label{la:llowersem}
Let $(X,d)$ be a metric space, and $\{\gamma_n\}_{n \in \mathbb{N}}$ be a sequence of curves into $X$. If $\gamma_n$ converges pointwise to a curve, $\gamma$, in $X$, then
\[
\mathcal{L}(\gamma) \leq \liminf_{n \to \infty} \mathcal{L}(\gamma_n)
\]
\end{proposition}
\begin{proof}
As discussed above, without loss of generality we can assume that all curves $\gamma_n: [0,1] \to X$.

Let $\eps > 0$ be arbitrary. Since
\[
\mathcal{L}(\gamma) = \sup_{p \in P} \sum_{\chgd{i \geq 1}} (d(\gamma(t_i),\gamma(t_{i-1}))
\]
where $P$ is the set of partitions of $[0,1]$, we can find a specific partition, $\mu=(t_0,t_1,\ldots,t_m)$, such that 

\begin{align*}
\mathcal{L}(\gamma)& < \left(\sum_{\chgd{t_i \in \mu, i \geq 1}} d(\gamma(t_i), \gamma(t_{i-1}))\right) + \frac{\eps}{2}.
\end{align*}

By pointwise convergence $\gamma_n(t) \xrightarrow{n\to \infty}  \gamma(t)$ for each fixed $t$, we can find $N \in \N$ such that 
\[
d(\gamma_n({t_i}), \gamma(t_i)) < \frac{\eps}{4m} \quad \forall i = 0,\ldots,m, \quad \forall n \geq N.
\]
Then, 
\begin{align*}
    d(\gamma(t_{i}),\gamma(t_{i-1}) &\leq d(\gamma(t_{i}),\gamma_n(t_{i})) + d(\gamma_n(t_{i}),\gamma_n(t_{i-1})) + d(\gamma_n(t_{i-1}),\gamma(t_{i-1}))\\
    &< \frac{\eps}{4m}  + d(\gamma_n(t_{i}),\gamma_n(t_{i-1})) + \frac{\eps}{4m}\\
    &= d(\gamma_n(t_{i}),\gamma_n(t_{i-1})) + \frac{\eps}{2m}.
\end{align*}
Thus,

\begin{align*}
\mathcal{L}(\gamma)& < \left(\sum_{\chgd{t_i \in \mu, i \geq 1}} d(\gamma_n(t_i), \gamma_n(t_{i-1})) \right)+ \frac{\eps}{2}+\frac{\eps}{2}.
\end{align*}

Finally, since 
\[
\mathcal{L}(\gamma_n) = \sup_{\mu \in \mathcal{P}} \sum_{\chgd{t_i \in \mu, i \geq 1}} (d(\gamma_n(t_i),\gamma_n(t_{i-1}))
\]
we have 
\[
\sum_{\chgd{t_i \in \mu, i \geq 1}} d(\gamma_n(t_i), \gamma_n(t_{i-1})) \leq \mathcal{L}(\gamma_n) .
\]

Thus we have shown,
\[
\mathcal{L}(\gamma) < \mathcal{L}(\gamma_n) + \eps, \quad \forall n \geq N.
\]
In particular 
\[
\mathcal{L}(\gamma) < \liminf_{n \to \infty} \mathcal{L}(\gamma_n) + \eps.
\]
This holds for any $\eps > 0$, letting $\eps \to 0$ we conclude
\[
\mathcal{L}(\gamma) \leq \liminf_{n \to \infty} \mathcal{L}(\gamma_n).
\]
\end{proof}

From \chgd{Arzel\'a}-Ascoli theorem, \Cref{th:arzelaascoli}, and the observations above we obtain the existence of shortest curves in the following sense. %Recall that a set $E$ is compact

\begin{theorem}\label{th:existencegeodesic:compact}
Let $(X,d)$ be any \emph{complete} metric space and $E \subset X$ be a compact set. Let $p \chgd{\neq} q \in E$ such that there exists a continuous curve $\gamma_0: [0,1] \to E$ of finite length $\mathcal{L}(\gamma_0) < \infty$ and $\gamma_0(0) = p$ and $\gamma_0(1) = q$. Then there exists a geodesic between $p$ and $q$, i.e. a curve $\gamma: [0,1] \to E$ such that $\gamma(0) = p$ and $\gamma(1) = q$ and such that 
\[
\mathcal{L}(\gamma) = \inf_{\tilde{\gamma}} \mathcal{L}(\tilde{\gamma})
\]
where the infimum is taken over all continuous curves $\tilde{\gamma}: [0,1] \to E$ with $\tilde{\gamma}(0) = p$ and $\tilde{\gamma}(1) = q$.
\end{theorem}
It is important to note that above the notion of ``shortest curve'' is with respect to $E$ not with respect to $X$, and this might lead to a different notion of what is a shortest curve. Take for example a compact banana-shaped set $E$ in $\R^3$. The straight line from top to bottom of the banana $E$ is likely to not lie within $E$, so it is not the shortest curve in $E$!
\begin{proof}[Proof of \Cref{th:existencegeodesic:compact}]
For simplicity we assume $X = E$. Since there exists one curve connecting $p$ and $q$ with finite length we have
\[
I := \inf_{\tilde{\gamma}} \mathcal{L}(\tilde{\gamma}) \in [0,\infty).
\]
Since there exists one curve connecting $p$ and $q$ there also must be a ``minimizing sequence''
\[
\gamma_k: [0,1] \to X \quad \text{ of finite length, $\mathcal{L}(\gamma_0) < \infty$, and $\gamma_k(0) = p$ and $\gamma_k(1) = q$}
\]
such that 
\[
\mathcal{L}(\gamma_k) \xrightarrow{k \to \infty} I.
\]
We may even assume that 
\[
I \leq \mathcal{L}(\gamma_k)\leq I+\frac{1}{k} \quad \forall k.
\]
By \Cref{pr:reparam} we may assume \chgd{without loss of generality} (otherwise use $\tilde{\gamma}_k$ instead of $\gamma_k$)
\[
|\gamma_k(x)-\gamma_k(y)| \leq \left(I+\frac{1}{k}\right) |x-y| \quad \forall x,y \in [0,1], \quad k \in \mathbb{N}.
\]
By \chgd{Arzel\'a}-Ascoli, Theorem \ref{th:arzelaascoli}, we may assume that we have uniform convergence to some continuous $\gamma: [0,1] \to X$, otherwise we could pass yet again to a subsequence.

Then, by lower semicontinuity of the length, Proposition~\ref{la:llowersem}, we have
\[
\mathcal{L}(\gamma) \leq \liminf_{k \to \infty} \mathcal{L}(\gamma_k)
\]
This means 
\[
I \leq \mathcal{L}(\gamma) \leq  \liminf_{k \to \infty} \mathcal{L}(\gamma_k) = I.
\]
So $\gamma$ is a shortest curve.
\end{proof}

% Lastly let us mention a famous theorem in Analysis: the Rademacher Theorem. It says that Lipschitz continuous maps into Euclidean space are almost everywhere differentiable:
% \begin{definition}[Sets of measure zero]
% A set $A \subset \R$ has Lebesgue measure zero if for any $\eps > 0$ there exists a countable collection of intervals $((x_i-r_i,x_i+r_i))_{i=1}^\infty$ such that 
% \begin{itemize}
%     \item $A \subset \bigcup_{i=1}^\infty (x_i-r_i,x_i+r_i)$, and
%     \item $\sum_{i=1}^\infty r_i < \eps.$
% \end{itemize}
% \end{definition}

% \begin{theorem}[Rademacher Theorem]\label{th:rademacher}
% Let $f:[a,b] \to \R^n$ be Lipschitz continuous, i.e. assume there exists $L \geq 0$ such that 
% \[
% |f(x)-f(y)|_{\R^n} \leq L |x-y| \quad \forall x,y \in [a,b].
% \]
% Then there exists a set of measure zero $\Sigma \subset [a,b]$, such that 
% \begin{itemize}
%     \item $f$ is differentiable for any $x \in [a,b] \setminus \Sigma$
%     \item on each of these points $x \in [a,b] \setminus \Sigma$ we have $|f'(x)| \leq L$
% \end{itemize}
% \end{theorem}

% For a proof of this difficult theorem see \cite[\textsection 3.1.2, Theorem 3.2]{EG15}.

\chgd{\section{Horizontal curves in the Heisenberg group}\label{s:horizontalcurves}
Recall that a differentiable curve $\gamma: [0,1] \to \R^3$ is called horizontal if \eqref{eq:horizontal} holds. In this section we compute important properties of horizontal curves that we will use in the proofs of both our main theorems.}

\begin{proposition}\label{th:derivative2}
If $\gamma \in C^2([0,1])$ %(\ToDo check: $\gamma \in C^1$) 
and \eqref{eq:horizontal} holds.
Then
\[
\lim_{s \to t} \frac{\frac{\gamma^3(t) - \gamma^3(s)}{t-s} + 2\left(\frac{\left(\gamma^2(t) - \gamma^2(s)\right)}{t-s}\gamma^1(s) - \frac{\left(\gamma^1(t) - \gamma^1(s)\right)}{t-s}\gamma^2(s)\right) }{t-s} =0. 
\]
The convergence rate is uniform in $t$.
\end{proposition}
\begin{proof}
\if
We want to show \[
\begin{split}
\lim_{s \to t} \left|\frac{\frac{\gamma^3(t) - \gamma^3(s)}{t-s} + 2\left(\frac{\left(\gamma^2(t) - \gamma^2(s)\right)}{t-s}\gamma^1(s) - \frac{\left(\gamma^1(t) - \gamma^1(s)\right)}{t-s}\gamma^2(s)\right) }{t-s} \right| = 0\\
\end{split}
\]
\fi
Since $\gamma$ is $C^2$, we have
\[
\gamma(s) = \gamma(t) + (s-t) \dot{\gamma}(t) + \frac{1}{2} \ddot{\gamma}(t) (s-t)^2 + o(|t-s|^2).
\]
and $o$ is uniform in the domain of $\gamma$.

Then,
\[
\begin{split}
  &\frac{\frac{\gamma^3(t) - \gamma^3(s)}{t-s} + 2\left(\frac{\left(\gamma^2(t) - \gamma^2(s)\right)}{t-s}\gamma^1(s) - \frac{\left(\gamma^1(t) - \gamma^1(s)\right)}{t-s}\gamma^2(s)\right) }{t-s}  \\
  =& \frac{\dot{\gamma}^3(t) - \frac{1}{2} \ddot{\gamma}^3(t) (t-s) + 2\left(\left( \dot{\gamma}^2(t) - \frac{1}{2} \ddot{\gamma}^2(t) (t-s)\right)\gamma^1(s) - \left(\dot{\gamma}^1(t) - \frac{1}{2} \ddot{\gamma}^1(t) (t-s)\right)\gamma^2(s)\right) }{t-s}
  +   o(1)\\
  =& \frac{- \frac{1}{2} \ddot{\gamma}^3(t) (t-s) + 2\left(\left(  - \frac{1}{2} \ddot{\gamma}^2(t) (t-s)\right)\gamma^1(s) - \left( - \frac{1}{2} \ddot{\gamma}^1(t) (t-s)\right)\gamma^2(s)\right) }{t-s}\\
  &+\frac{\dot{\gamma}^3(t) + 2\left(\left( \dot{\gamma}^2(t) \right)\gamma^1(s) - \left(\dot{\gamma}^1(t) \right)\gamma^2(s)\right) }{t-s}\\
    &+   o(1)\\
      =& - \frac{1}{2} \ddot{\gamma}^3(t)  + 2\left( - \frac{1}{2} \ddot{\gamma}^2(t) \gamma^1(t) + \frac{1}{2} \ddot{\gamma}^1(t) \gamma^2(t)\right) +o(1)\\
  &+\frac{\dot{\gamma}^3(t) + 2\left(\left( \dot{\gamma}^2(t) \right)\gamma^1(s) - \left(\dot{\gamma}^1(t) \right)\gamma^2(s)\right) }{t-s}\\
    &+   o(1)
 \
 \end{split}
\]
\chgd{We define}
\[
f(s) := \dot{\gamma}^3(t) + 2\left(\left( \dot{\gamma}^2(t) \right)\gamma^1(s) - \left(\dot{\gamma}^1(t) \right)\gamma^2(s)\right)
\]
and, we observe that by horizontality, $f(t)=0$. Thus 
\[
\frac{f(s)}{t-s} = -\frac{f(s)-f(t)}{s-t} =-f'(t)+o(1)
\]
Then, we have 
\[
f'(s) = 2\left(\left( \dot{\gamma}^2(t) \right)\dot{\gamma}^1(s) - \left(\dot{\gamma}^1(t) \right)\dot{\gamma}^2(s)\right)
\]
So 
\[
f'(t) =  2\left(\left( \dot{\gamma}^2(t) \right)\dot{\gamma}^1(t) - \left(\dot{\gamma}^1(t) \right)\dot{\gamma}^2(t)\right)
\]

Consequently,
\[
\begin{split}
  &\frac{\frac{\gamma^3(t) - \gamma^3(s)}{t-s} + 2\left(\frac{\left(\gamma^2(t) - \gamma^2(s)\right)}{t-s}\gamma^1(s) - \frac{\left(\gamma^1(t) - \gamma^1(s)\right)}{t-s}\gamma^2(s)\right) }{t-s}  \\
      =& - \frac{1}{2} \ddot{\gamma}^3(t)  + 2\left( - \frac{1}{2} \ddot{\gamma}^2(t) \gamma^1(t) + \frac{1}{2} \ddot{\gamma}^1(t) \gamma^2(t)\right) +o(1)\\
  &-\left (   2\left(\left( \dot{\gamma}^2(t) \right)\dot{\gamma}^1(t) - \left(\dot{\gamma}^1(t) \right)\dot{\gamma}^2(t)\right)\right )+ o(1)\\
    &+   o(1)\\
        =& - \frac{1}{2} \frac{d}{dt}\brac{ \dot{\gamma}^3(t)  + 2\left( \dot{\gamma}^2(t) \gamma^1(t)- \dot{\gamma}^1(t) \gamma^2(t)\right) }\\
        & - \frac{1}{2} \brac{  - 2\left( \dot{\gamma}^2(t) \dot{\gamma}^1(t)- \dot{\gamma}^1(t) \dot{\gamma}^2(t)\right) }\\
  &-\left (  + 2\left(\left( \dot{\gamma}^2(t) \right)\dot{\gamma}^1(t) - \left(\dot{\gamma}^1(t) \right)\dot{\gamma}^2(t)\right)\right )\\
    &+   o(1)\\
 \
 =&       0\\
          &-\left (  + 1\left(\left( \dot{\gamma}^2(t) \right)\dot{\gamma}^1(t) - \left(\dot{\gamma}^1(t) \right)\dot{\gamma}^2(t)\right)\right )\\
    &+   o(1)\\
    =&\chgd{o(1)}.
 \end{split}
\]
Then
\[
\lim_{s \to t} \frac{\frac{\gamma^3(t) - \gamma^3(s)}{t-s} + 2\left(\frac{\left(\gamma^2(t) - \gamma^2(s)\right)}{t-s}\gamma^1(s) - \frac{\left(\gamma^1(t) - \gamma^1(s)\right)}{t-s}\gamma^2(s)\right) }{t-s} =0. 
\]
as desired.
%\Josh{the above part is fine but the below used the wrong metric. should be right now}

% To get the theorem, let us look at two directions

% \begin{lemma}
% Assumptions as in the theorem, then 
% \[
% \mathcal{L}(\gamma) \geq \int_{[0,1]} \left (\left (\dot{\gamma}^1(t)\right )^4 + \left (\dot{\gamma}^2(t) \right)^4 \right )^{\frac{1}{4}} dt
% \] 
% \end{lemma}
% \begin{proof}

% Observe that 
% \[
% \frac{d_K(\gamma(t_{i+1}),\gamma(t_i))}{|t_{i+1}-t_i|} \geq \frac{|\gamma^1(t_{i+1})-\gamma^1(t_i)|^4 + |\gamma^2(t_{i+1})-\gamma^2(t_i)|^4}{|t_{i+1}-t_i|} \xrightarrow{|t_{i+1}-t_i| \to 0} \left (\left (\dot{\gamma}^1(t)\right )^4 + \left (\dot{\gamma}^2(t) \right)^4 \right )^{\frac{1}{4}}.
% \]
% \ToDo details
% \end{proof}

Then,

%\[
%\lim_{s \to t} \left|\frac{\gamma^1(t)-\gamma^1(s)}{t-s}\right|^4 + \left|\frac{\gamma^2(t)-\gamma^2(s)}{t-s}\right|^4
%= \left|\dot{\gamma}^1(t)\right|^4 + \left|\dot{\gamma}^2(t)\right|^4
%\]

\begin{align*}
&\lim_{s \to t} \frac{\left[\brac{|\gamma^1(t)-\gamma^1(s)|^2 + |\gamma^2(t)-\gamma^2(s)|^2} +  \left|\gamma^3(t) - \gamma^3(s) + 2\left(\left(\gamma^2(t) - \gamma^2(s)\right)\gamma^1(s) - \left(\gamma^1(t) - \gamma^1(s)\right)\gamma^2(s)\right)\right| \right]^2}{|t-s|^4}  \\
&=\lim_{s \to t} \frac{\brac{|\gamma^1(t)-\gamma^1(s)|^2 + |\gamma^2(t)-\gamma^2(s)|^2}^2}{|t-s|^4}\\
=& \lim_{s \to t} \frac{|\gamma^1(t)-\gamma^1(s)|^4 + 2|\gamma^1(t)-\gamma^1(s)|^2 |\gamma^2(t)-\gamma^2(s)|^2 +  |\gamma^2(t)-\gamma^2(s)|^4}{|t-s|^4}\\
=& \lim_{s \to t} \left(\frac{|\gamma^1(t)-\gamma^1(s)|}{|t-s|}\right)^4 + 2\left(\frac{|\gamma^1(t)-\gamma^1(s)|}{|t-s|}\right)^2 \left(\frac{|\gamma^2(t)-\gamma^2(s)|}{|t-s|}\right)^2 + \left(\frac{|\gamma^2(t)-\gamma^2(s)|}{|t-s|}\right)^4\\
=& \dot{\gamma}^1(t)^4 + 2\dot{\gamma}^1(t)^2 \dot{\gamma}^2(t)^2 + \dot{\gamma}^4(t) = \left( \dot{\gamma}^1(t)^2 + \dot{\gamma}^2(t)^2 \right)^2   
\end{align*}
and the convergence is uniformly in $t$ by the above considerations.
\end{proof}

From \Cref{th:derivative2} we readily obtain
\begin{corollary}\label{co:diffofgamma}
If $\gamma \in C^2([0,1],\R^3)$ and \eqref{eq:horizontal} holds 
\[
\frac{d_K(\gamma(t),\gamma(s))}{|t-s|} \xrightarrow{s \to t} \sqrt{ \dot{\gamma}^1(t)^2 + \dot{\gamma}^2(t)^2    }
%\frac{d_K(\gamma(t),\gamma(s))}{|t-s|} \xrightarrow{s \to t} \sqrt[4]{|\dot{\gamma}^1(t)|^4 + |\dot{\gamma}^2(t)|^4}
\]
The convergence is uniform in $t$. In particular we have 
\[
 \mathcal{L}_{K}(\gamma) < \infty.
\]
\end{corollary}

\section{Existence of shortest curves in the Heisenberg group} \label{s:heisenbergcurves}
In this section we want to show Theorem~\ref{th:existencegeodesic}.

Of course we would like to apply Theorem~\ref{th:existencegeodesic:compact}, however we need to be careful with the compactness assumption in that theorem, since $\mathbb{H}_1$ is not compact. However, one could justifiably believe that any curve $\gamma: [0,1] \to \mathbb{H}_1$ which goes too far away from $p$ and $q$ is not a good candidate for shortest curve. We need to quantify this and for this we compare the \chgd{Kor\'anyi} metric locally with the Euclidean metric.

\begin{lemma}\label{la:compactsets}
Let $K \subset \R^3$ be compact (in the sense of the Euclidean metric). Then $K \subset \mathbb{H}_1$ is compact (in the sense of the \chgd{Kor\'anyi} metric).
\end{lemma}
\begin{proof}
Since $K$ is compact as Euclidean set $\R^3$ it is bounded and thus there must be some $\Lambda > 0$ such that 
\[
\max \{|p_1|,|p_2|, |p_2|\} < \Lambda \quad \forall p = (p_1,p_2,p_3) \in K.
\]
Using repeatedly Young's inequality $2ab \leq a^2+b^2$ we find that for $p,q \in K$
\[
\begin{split}
d_K(q,p) =& (\brac{|p_1-q_1|^2 + |p_2-q_2|^2}^2 + \left |p_3 - q_3 + 2(p_2 q_1 - p_1 q_2)\right |^2)^{\frac{1}{4}} \\
\leq& (\brac{|p_1-q_1|^2 + |p_2-q_2|^2}^2 + 2\left |p_3 - q_3 \right |^2+ 2\left |2(p_2 q_1 - p_1 q_2)\right |^2)^{\frac{1}{4}} \\
=& (\brac{|p_1-q_1|^2 + |p_2-q_2|^2}^2 + 2\left |p_3 - q_3 \right |^2+ 2\left |2(p_2 -q_2)q_1 + (q_1- p_1) q_2)\right |^2)^{\frac{1}{4}} \\
\leq& (\brac{|p_1-q_1|^2 + |p_2-q_2|^2}^2 + 2\left |p_3 - q_3 \right |^2+ 8\left ( |p_2 -q_2|\Lambda + |q_1- p_1| \Lambda )\right )^2)^{\frac{1}{4}} \\
\end{split}
\]
We conclude that for each $\eps > 0$ there exists $\delta > 0$ such that if $p,q \in K$ and $|p-q| < \delta$ (in the Euclidean sense) then $d_{K}(p,q) < \eps$. 

In particular any (Euclidean) converging sequence in $K$ also converges in the sense of the \chgd{Kor\'anyi} metric $d_K$. Thus $K$ is also compact in the \chgd{Kor\'anyi} sense.
\end{proof}

The following lemma shows that ``far away'' in the Euclidean sense implies ``far away'' in the \chgd{Kor\'anyi} sense.
\begin{lemma}\label{la:largeeuclidlargeK}
Fix $q \in \R^3$. For any $\Lambda > 0$ there exists $\Theta >0$ such that the following is true:
if for some $p \in \R^3$ we have  
\[
|p-q| > \Theta
\]
then 
\[
d_{K}(p,q) > \Lambda.
\]
\end{lemma}
\begin{proof}
Observe that for any $p,q \in \R^3$
\begin{align*}
    d_K(p,q)^{\chgd{2}} \geq \left |p_3 - q_3 + 2(p_2 q_1 - p_1 q_2)\right | &= \left|p_3 - q_3 + 2((p_2-q_2) q_1  + q_1q_2 - (p_1-q_1) q_2 -q_1q_2)\right |\\
    &= \left|p_3 - q_3 + 2((p_2-q_2) q_1 - (p_1-q_1) q_2\right |\\
    &\geq \left(|p_3 - q_3| - 2\left|q_1(p_2-q_2) - q_2(p_1-q_1)\right|\right)\\
    &\geq \left(|p_3 - q_3| - 2\left(|q_1||p_2-q_2| + |q_2||p_1-q_1|\right)\right)\\
    &\geq \left(|p_3 - q_3| - 2\left(|q_1||p_2-q_2| + |q_2||p_1-q_1|\right)\right)\\
\end{align*}

Now fix $q =(q_1,q_2,q_3) \in \R^3$ and $\Lambda > 0$ and set 
\[
\Gamma := |q_1| + |q_2|.
\]
Take $\Theta > 0$ so that the following conditions are satisfied: $\Theta > \sqrt{3} \Lambda$ and $\frac{1}{\sqrt{3}}\Theta - 2\Gamma\, \Lambda > \Lambda^2$.

Now take $p = (p_1,p_2,p_3)\in \R^3$ such that
\[
|p-q| > \Theta.
\]
Then
\[
\max\{|p_1-q_1|,|p_2-q_2|,|p_3-q_3|\} > \frac{1}{\sqrt{3}} \Theta.
\]
\chgd{Then either}
\[
\max\{|p_1-q_1|,|p_2-q_2|\} > \Lambda
\]
or 
\[
|p_3-q_3| > \frac{1}{\sqrt{3}}\Theta.
\]
From the above estimates we have 
\[
d_K(p,q) \geq \max\left \{|p_1-q_1|,|p_2-q_2|, \left(|p_3 - q_3| - 2\left(|q_1||p_2-q_2| + |q_2||p_1-q_1|\right)\right)^\frac{1}{2}\right \} 
\]
In the case that $\max\{|p_1-q_1|,|p_2-q_2|\} > \Lambda$ we conclude that
\[
d_{K}(p,q) > \Lambda,
\]
and we are done. 
If on the other hand both $|p_1-q_1|$ or $|p_2-q_2| < \Lambda$ then we have $|p_3-q_3| > \frac{1}{\sqrt{3}}\Theta$ and thus
\[
\begin{split}
d_K(p,q)^2 \geq& |p_3 - q_3| - 2\left(|q_1||p_2-q_2| + |q_2||p_1-q_1|\right) \\
\geq & \frac{1}{\sqrt{3}}\Theta - 2\Gamma\, \Lambda  
\end{split}
\]
Again in this case, by the choice of $\Theta$ we find that 
\[
d_K(p,q)^2 > \Lambda^2,
\]
and we can conclude $d_K(p,q) > \Lambda$ as desired.
\end{proof}

\begin{proof}[Proof of \Cref{th:existencegeodesic}]
Fix $p,q \in \R^3$. There exists a smooth horizontal curve $\tilde{\gamma}$ connecting $p$ and $q$, take for example the $\mathcal{L}_{cc}$-geodesic from  \cite{HZ}, and in view of \Cref{co:diffofgamma} $\tilde{\gamma}$ has finite length: $\mathcal{L}_{K}(\tilde{\gamma}) < \infty$. 

Let $R > 0$ such that for any $r \in \R^3$ with $|p-r| > R$ we have in view of \Cref{la:largeeuclidlargeK}
\[
d_{K}(p,r) > \mathcal{L}_{K}(\tilde{\gamma}).
\]
This implies that any continuous curve $\gamma: [0,1] \to \R^3$ with $\gamma(0) = p$ and $\gamma(1) = q$ and $|\gamma(t)-p| > R$ for any $t \in (0,1)$ we have 
\[
\mathcal{L}_\chgd{K}(\gamma) > \mathcal{L}_\chgd{K}(\tilde{\gamma}).
\]
Set $E := \{r \in \R^3: |r-p| \leq R\}$ which is a compact set in the Euclidean sense, and thus in view of Lemma~\ref{la:compactsets} also in the \chgd{Kor\'anyi} sense. Then we have shown that 
\[
\inf_{\gamma: [0,1] \to E} \mathcal{L}_K(\gamma) =\inf_{\gamma: [0,1] \to \R^3} \mathcal{L}_K(\gamma),
\]
where both infima are taken over continuous curves $\gamma$ with $\gamma(0) = p$ and $\gamma(1) = q$. Now we can finally apply \Cref{th:existencegeodesic:compact}. Thus, there is a shortest curve between $p$ and $q$.
\end{proof}

% \begin{definition} Lipschitz continuity

% For a curve $\gamma: [0,1] \to X$, with $X$ a metric space we define
% \begin{enumerate}
%     \item Lipschitz curve in the \chgd{Kor\'anyi} sense:
% \[
% d_{K}(\gamma(s),\gamma(t)) \leq L |s-t|
% \]

% \item Lipschitz curve in the Euclidean sense 
% \[
% |\gamma(s)-\gamma(t)| \leq L |s-t|
% \]
% where $s,t \in [0,1]$
% \end{enumerate}
% \end{definition}

% \begin{lemma}
% If $\gamma$ is Lipschitz in the \chgd{Kor\'anyi} sense, then $\gamma$ is Lipschitz in the Euclidean sense
% \end{lemma}
% \begin{proof}

% 
% \end{proof}

\section{Length of Curves in the Heisenberg group -- Proof of Theorem~\ref{th:lengthsame}}\label{s:proofofmainthm}

In this section we show that 
\[
\mathcal{L}_{cc}(\gamma) = \mathcal{L}_{K}(\gamma),
\]
whenever $\gamma \in C^2$ is a horizontal curve, i.e. whenever $\gamma$ satisfies \eqref{eq:horizontal}.

\begin{proof}[Proof of \Cref{th:lengthsame}]
From \eqref{eq:horizontal} in particular,
\[
\frac{d}{dt} \left (\dot{\gamma}^3(t) + 2 \left (\dot{\gamma}^2(t)\gamma^1(t) - \dot{\gamma}^1(t) \gamma^2(t)\right ) \right )= 0
\]
We apply \Cref{th:derivative2} and obtain
\begin{align*}
&\lim_{s \to t} \left(\frac{d_K(\gamma(t),\gamma(s))}{|t-s|}\right)^4
%= \left( \dot{\gamma}^1(t)^2 + \dot{\gamma}^2(t)^2 \right)^2   
=\lim_{s \to t} \frac{\brac{|\gamma^1(t)-\gamma^1(s)|^2 + |\gamma^2(t)-\gamma^2(s)|^2}^2}{|t-s|^4}\\
=& \lim_{s \to t} \frac{|\gamma^1(t)-\gamma^1(s)|^4 + 2|\gamma^1(t)-\gamma^1(s)|^2 |\gamma^2(t)-\gamma^2(s)|^2 +  |\gamma^2(t)-\gamma^2(s)|^4}{|t-s|^4}\\
=& \lim_{s \to t} \left(\frac{|\gamma^1(t)-\gamma^1(s)|}{|t-s|}\right)^4 + 2\left(\frac{|\gamma^1(t)-\gamma^1(s)|}{|t-s|}\right)^2 \left(\frac{|\gamma^2(t)-\gamma^2(s)|}{|t-s|}\right)^2 + \left(\frac{|\gamma^2(t)-\gamma^2(s)|}{|t-s|}\right)^4\\
=& \dot{\gamma}^1(t)^4 + 2\dot{\gamma}^1(t)^2 \dot{\gamma}^2(t)^2 + \dot{\gamma}^4(t) = \left( \dot{\gamma}^1(t)^2 + \dot{\gamma}^2(t)^2 \right)^2   
\end{align*}

Then, taking the fourth root,
\[
\lim_{s \to t} \frac{d_K(\gamma(t),\gamma(s))}{|t-s|} = \sqrt{ \dot{\gamma}^1(t)^2 + \dot{\gamma}^2(t)^2    }
\]
\end{proof}

\begin{proof}[Proof of $\mathcal{L}_{cc}(\gamma) = \mathcal{L}_{K}(\gamma)$ if $\gamma$ is horizontal]
%Use the Riemann sum theorem to show the result for $C^2$-maps \ToDo
%Want to show:
%\[
%\sup_{p} {\sum \frac{d_K(\gamma(t_i), \gamma(t_{i-1}))}{|t_i-t_{i-1}|}} |t_i-t_{i-1}| = \int %\sqrt{|\dot{\gamma}^1(t)|^2 + |\dot{\gamma}^2(t)|^2} \: dt
%\]
%where $p$ is a partition of $[0,1]$.\\
From Corollary 19, we see that 
\[
\lim_{s \to t} \frac{d_K(\gamma(t),\gamma(s))}{|t-s|} = \sqrt{ \dot{\gamma}^1(t)^2 + \dot{\gamma}^2(t)^2    }
\]
\chgd{uniformly in $t$.}
%By the fundamental theorem of calculus,
%\[
%\gamma(t_i) - \gamma(t_{i-1}) = \dot \gamma(t_{i-1})(t_i-t_{i-1}) + o(|t_{i-1}-t_i|)
%\]
%where
%\[
%o(|t_{i-1}-t_i|) \leq \eps |t_{i-1}-t_i|
%\]
Then, from the above limit, given some $\eps > 0$ choose $\delta>0$, such that when $|t_i-t_{i-1}|<\delta$, we have
\[
\left |\frac{d(\gamma(t_i),\gamma(t_{i-1}))}{|t_i-t_{i-1}|} -  \sqrt{|\dot{\gamma}^1(t_i)|^2 + |\dot{\gamma}^2(t_i)|^2} \right | < \eps.
\]
Multiplying by $|t_i-t_{i-1}|$,
\[
\left |\frac{d(\gamma(t_i),\gamma(t_{i-1}))}{|t_i-t_{i-1}|}|t_i-t_{i-1}| -  \sqrt{|\dot{\gamma}^1(t_i)|^2 + |\dot{\gamma}^2(t_i)|^2}|t_i-t_{i-1}| \right | < \eps|t_i-t_{i-1}|.
\]
Now, let $\mathcal{P}$ be the set of partitions of $[0,1]$ such that for any $\mu \in \mathcal{P}$, we have $|t_i-t_{i-1}| < \delta$ for each $t_i$ in $\mu$. % (this is always possible, by refinement property of the integral). Taking the supremum over $\mathcal{P}$.
Then, for a given $\mu \in \mathcal{P}$,
\begin{align*}
\sum_{\chgd{t_i \in \mu, i \geq 1}} \left |\frac{d(\gamma(t_i),\gamma(t_{i-1}))}{|t_i-t_{i-1}|}|t_i-t_{i-1}| -  \sqrt{|\dot{\gamma}^1(t_i)|^2 + |\dot{\gamma}^2(t_i)|^2}|t_i-t_{i-1}| \right | &<\eps \underbrace{ \sum_{\chgd{t_i \in \mu, i \geq 1}} |t_i-t_{i-1}|}_{=1}
&= \eps
\end{align*}
So, we get

\begin{align*}
&\left |\sum_{\chgd{t_i \in \mu, i \geq 1}} \frac{d(\gamma(t_i),\gamma(t_{i-1}))}{|t_i-t_{i-1}|} |t_i-t_{i-1}| - \sum_{\chgd{t_i \in \mu, i \geq 1}} \sqrt{|\dot{\gamma}^1(t_i)|^2 |\dot{\gamma}^2(t_i)|^2} |t_i-t_{i-1}| \right | \\
&=\left |\sum_{\chgd{t_i \in \mu, i \geq 1}} d(\gamma(t_i),\gamma(t_{i-1})) - \sum_{\chgd{t_i \in \mu, i \geq 1}} \sqrt{|\dot{\gamma}^1(t_i)|^2 + |\dot{\gamma}^2(t_i)|^2} |t_i-t_{i-1}| \right |<  \eps 
\end{align*}

Then,
\begin{align*}
\eps &> \sup_{\mu \in \mathcal{P}} \left |\sum_{\chgd{t_i \in \mu, i \geq 1}} d(\gamma(t_i),\gamma(t_{i-1})) - \sum_{\chgd{t_i \in \mu, i \geq 1}} \sqrt{|\dot{\gamma}^1(t_i)|^2 + |\dot{\gamma}^2(t_i)|^2} |t_i-t_{i-1}| \right |  \\
& \geq \sup_{\mu \in \mathcal{P}} \left |\sum_{\chgd{t_i \in \mu, i \geq 1}} d(\gamma(t_i),\gamma(t_{i-1})) \right|-\sup_{\mu\in \mathcal{P}}\left| \sum_{\chgd{t_i \in \mu, i \geq 1}} \sqrt{|\dot{\gamma}^1(t_i)|^2+ |\dot{\gamma}^2(t_i)|^2} |t_i-t_{i-1}| \right |\\
&= \mathcal{L}_K(\gamma) - \sup_{\mu\in \mathcal{P}}\left| \sum_{\chgd{t_i \in \mu, i \geq 1}} \sqrt{|\dot{\gamma}^1(t_i)|^2 + |\dot{\gamma}^2(t_i)|^2} |t_i-t_{i-1}| \right |
\end{align*}

Similarly,
\begin{align*}
\eps &> \sup_{\mu \in \mathcal{P}} \left |\sum_{\chgd{t_i \in \mu, i \geq 1}} d(\gamma(t_i),\gamma(t_{i-1})) - \sum_{\chgd{t_i \in \mu, i \geq 1}} \sqrt{|\dot{\gamma}^1(t_i)|^2 + |\dot{\gamma}^2(t_i)|^2} |t_i-t_{i-1}| \right |  \\
& \geq \sup_{\mu\in \mathcal{P}}\left| \sum_{\chgd{t_i \in \mu, i \geq 1}} \sqrt{|\dot{\gamma}^1(t_i)|^2+ |\dot{\gamma}^2(t_i)|^2} |t_i-t_{i-1}| \right |-\sup_{\mu \in \mathcal{P}} \left |\sum_{\chgd{t_i \in \mu, i \geq 1}} d(\gamma(t_i),\gamma(t_{i-1})) \right|\\
&=  \sup_{\mu\in \mathcal{P}}\left| \sum_{\chgd{t_i \in \mu, i \geq 1}} \sqrt{|\dot{\gamma}^1(t_i)|^2 + |\dot{\gamma}^2(t_i)|^2} |t_i-t_{i-1}| \right |-\mathcal{L}_K(\gamma)
\end{align*}
That is,
\begin{align*}
\eps > \left | \sup_{\mu\in \mathcal{P}}\left| \sum_{\chgd{t_i \in \mu, i \geq 1}} \sqrt{|\dot{\gamma}^1(t_i)|^2 + |\dot{\gamma}^2(t_i)|^2} |t_i-t_{i-1}| \right |-\mathcal{L}_K(\gamma) \right |
\end{align*}

So, we have 
\[
\begin{split}
&\left |\mathcal{L}_K(\gamma) - \sup_{\mu\in \mathcal{P}}\left| \sum_{\chgd{t_i \in \mu, i \geq 1}} \sqrt{|\dot{\gamma}^1(t_i)|^2 +  |\dot{\gamma}^2(t_i)|^2} |t_i-t_{i-1}| \right | \right |\\
\leq& \left |\mathcal{L}_K(\gamma) - \sup_{\mu\in \mathcal{P}} \sum_{\chgd{t_i \in \mu, i \geq 1}} \sqrt{|\dot{\gamma}^1(t_i)|^2 + |\dot{\gamma}^2(t_i)|^2} |t_i-t_{i-1}|\right |  < \eps
\end{split}
\]
Note that, since $\chgd{\dot{\gamma}}$ is continuous, \chgd{the function}
\[
\chgd{\sqrt{|\dot{\gamma}^1(t_i)|^2 + |\dot{\gamma}^2(t_i)|^2}}
\]
is \chgd{continuous and hence} integrable. So,
\[
\sum_{\chgd{t_i \in \mu, i \geq 1}} \sqrt{|\dot{\gamma}^1(t_i)|^2 + |\dot{\gamma}^2(t_i)|^2} |t_i-t_{i-1}|
\]
is a Riemann Sum, and
\[
\sup_{\mu \in \mathcal{P}} \sum_{\chgd{t_i \in \mu, i \geq 1}} \sqrt{|\dot{\gamma}^1(t_i)|^2 + |\dot{\gamma}^2(t_i)|^2} |t_i-t_{i-1}|
 = \int_{0}^{1}  \sqrt{|\dot{\gamma}^1(t)|^2 + |\dot{\gamma}^2(t)|^2} dt
\]
Then,
\[
\left |\mathcal{L}_K(\gamma) - \int_{0}^{1}  \sqrt{|\dot{\gamma}^1(t)|^2 + |\dot{\gamma}^2(t)|^2} dt\right | < \eps
\]
This holds for any $\eps > 0$, so letting $\eps \to 0$ we conclude 
\[
\mathcal{L}_K(\gamma) = \int_{0}^{1}  \sqrt{|\dot{\gamma}^1(t)|^2 + |\dot{\gamma}^2(t)|^2} dt.
\]
This proves  $\mathcal{L}_{cc}(\gamma) = \mathcal{L}_{K}(\gamma)$ which in particular implies \Cref{th:lengthsame}.
\end{proof}

\subsection*{Acknowledgement}
This is part of a Undergraduate Research Project with Dr. Schikorra. Funding was provided by NSF Career DMS-2044898.

\bibliographystyle{abbrv}
\bibliography{bib}

\end{document}